\documentclass[a4paper, 12pt]{article}
\usepackage{tgheros}
\usepackage{amsthm}
\newtheorem{thm}{Theorem}

\newtheorem{lemma}[thm]{Lemma}

\usepackage{amssymb,amsmath}

\newtheorem*{theorem*}{Theorem}

\DeclareMathOperator{\F}{\mathbb{F}}

\begin{document}
\baselineskip=16.3pt
\parskip=14pt

\begin{center}
\section*{On Galois groups of linearized polynomials related to the general linear group of prime degree}

{\large 
Rod Gow and Gary McGuire
 \\ { \ }\\
School of Mathematics and Statistics\\
University College Dublin\\
Ireland}
\end{center}

 \subsection*{Abstract}
 
Let $L(x)$ be any $q$-linearized polynomial with coefficients in $\F_q$, of degree $q^n$.
We consider the Galois group of $L(x)+tx$ over $\F_q(t)$, where $t$ is transcendental
over $\F_q$.
We prove that when $n$ is a prime, the Galois group is always $GL(n,q)$,
except when $L(x)=x^{q^n}$.
 Equivalently, we prove that the arithmetic monodromy group of $L(x)/x$ is  $GL(n,q)$.
 
\section{Introduction}
  
Let $p$ be a prime number, and let $q$ be a positive power of $p$.
Let $\F_q$ denote the finite field with $q$ elements.
Let $F$ be a field of characteristic $p$,
and assume that $F$ contains $\F_q$.
A $q$-linearized polynomial over $F$  is a polynomial of the form 
\begin{equation}\label{lin1}
L(x)=a_0x+a_1x^q+a_2x^{q^2}+\cdots+a_n x^{q^n} \in F[x].
\end{equation}
If $a_n\not=0$ we say that $n$ is the $q$-degree of $f$.
We will usually say $q$-polynomial instead of $q$-linearized polynomial.

The set of roots of a $q$-polynomial $L$ forms an $\F_q$-vector space,
which is contained in a splitting field of $L$. We make this statement more precise in the following lemma, which is straightforward to prove (see \cite{SL} for example).

\begin{lemma} \label{vector_space}
Let $F$ be a field of prime characteristic $p$ that contains $\F_q$. Let $L$ be a $q$-polynomial of $q$-degree $n$ in $F[x]$, with
\[
L=a_0x+a_1x^q+a_2x^{q^2}+\cdots+a_n x^{q^n}.
\]
Let $E$ be a splitting field for $L$ over $F$ and let $V$ be the set of roots of $L$ in $E$. 
Let $\Gamma$ be the Galois group of $E$
over $F$. Suppose that $a_0\neq 0$. 
Then $V$ is an $\F_q$-vector space of dimension $n$ and $\Gamma$ is naturally a subgroup
of $GL(n,q)$.
\end{lemma}

In this statement we are denoting the general linear group of invertible $n\times n$ matrices over $\F_q$ by $GL(n,q)$.
The affine general linear group, denoted $AGL(n,q)$, is the semidirect product of
$GL(n,q)$ with $\F_q^n$, under the natural action.
The group $\Gamma L (n,q)$ is the semidirect
product of $GL(n,q)$ with the automorphism group of  $\F_{q}$, where the
automorphisms act entrywise on an element of $GL(n,q)$.
%We call $V$ the ($\F_q$-vector) space of roots of $L$.

In this paper we will take $F$ to be $\F_q(t)$, where $t$ is transcendental over $\F_q$.
We consider polynomials of the form $L(x)+tx \in \F_q(t)[x]$ where $L(x)$ has coefficients in $\F_q$.
We will prove the following theorem. 

\begin{thm}\label{main}
Let $q$ be odd, and let $n$ be an odd prime. 
Let $L(x)\in \F_q[x]$ be a monic $q$-polynomial of $q$-degree $n$.
Then the Galois group of $L(x)+tx \in \F_q(t)[x]$ is isomorphic to $GL(n,q)$, unless $L(x)=x^{q^n}$,
in which case the Galois group is isomorphic to $\Gamma L (1,q^n)$.
\end{thm}

The proof of Theorem \ref{main} utilizes several results and methods from the literature.
We will present the background and statements of these results in Section 2.
Then we present the proof of Theorem \ref{main} in Section 3.
In Section \ref{char2} we present a strong partial result in characteristic 2.

Throughout the paper, we let $L(x)\in \F_q[x]$ be a monic $q$-polynomial of $q$-degree $n$.
We consider $L(x)+tx \in \F_q(t)[x]$, where $t$ is transcendental over $\F_q$.
Let $\Gamma$ be the Galois group of $L(x)+tx$ over $\F_q(t)$.
Since 0 is a root of $L(x)+tx$,
and all elements of $\Gamma$ permute the roots and act linearly,
$\Gamma$ fixes 0. 
Therefore $\Gamma$ is also the Galois group of $(L(x)+tx)/x=L(x)/x+t$.
We will sometimes consider this polynomial, which has degree $q^n-1$.

Let $f(x)\in \F_q[x]$ be any separable polynomial and let $t$ be transcendental over $\F_q$.
The Galois group of $f(x)+t$ over $\F_q(t)$ is called the arithmetic monodromy group of $f$.
If $\overline{\F_q}$ denotes the algebraic closure of $\F_q$ then the
Galois group of $f(x)+t$ over $\overline{\F_q}(t)$ is called the geometric monodromy group of $f$.
It is a standard fact that
the arithmetic monodromy group   contains the geometric monodromy group 
 as a normal subgroup, with cyclic quotient.
It is still an open problem to decide which finite groups occur as monodromy groups.
By the previous paragraph, this paper studies the 
monodromy groups of polynomials of the form $L(x)/x$, where $L(x)\in \F_q[x]$ is a $q$-polynomial.
Theorem \ref{main} essentially states that the arithmetic monodromy group of
$L(x)/x$ is $GL(n,q)$.

Previous work on realizing groups as monodromy/Galois groups includes 
finding a particular polynomial with a certain group. 
 A result of Abhyankar \cite{Ab} shows that the
polynomial $x^{q^n}+x^q+tx$ has Galois group $GL(n,q)$ over $\F_q(t)$  for all positive $n$.
Our result says that when $n$ is a prime, $x^{q^n}+x^q$ can be replaced by 
any $q$-polynomial of $q$-degree $n$ with
coefficients in $\F_q$, apart from $x^{q^n}$, and the Galois group will not change.

Other previous work includes Elkies \cite{E} who
shows (among other things) that the Galois group of 
a generic $q$-polynomial
 \eqref{lin1},
where  $a_0, \ldots , a_n$ are algebraically independent transcendentals,
is $GL(n,q)$ over $\F_q(a_0, \ldots , a_n)$.
Elkies also shows that the Galois group of a
generic  affine polynomial of the form $L(x)+t$, where $L(x)$ is given by \eqref{lin1} 
and  $a_0, \ldots , a_n, t$ are algebraically independent transcendentals,
is $AGL(n,q)$ over $\F_q(a_0, \ldots , a_n, t)$.
One may ask what happens when some of the transcendentals are specialized to be
in $\F_q$. It is easy to see that the Galois group of the resulting polynomial
will be a subgroup of the original group, 
assuming the specialization is separable,
but it is not  clear whether it is \emph{equal}
to the original group.

M\"oller \cite{M} answers this for affine polynomials.
He shows (among other things) that an affine polynomial 
$L(x)+t$ still has Galois group $AGL(n,q)$, where $L(x)$ is a $q$-polynomial with
coefficients in $\F_q$ having $q$-degree $n$, and $n$ is prime.
In this article we answer the question for $GL(n,q)$.
We consider a generic
$q$-polynomial \eqref{lin1} 
and we  specialize all coefficients to be in $\F_q$, except the
coefficient of $x$. We will prove that the Galois group remains 
equal to $GL(n,q)$, when $n$ is an odd prime.
This is equivalent to showing that the arithmetic monodromy group of
$L(x)/x$ is $GL(n,q)$.
Our proof methods use some of the same results from finite group theory as used in \cite{M},
however our proof overall is different, especially in the use of \cite{GMcG}.

\section{Background}

\subsection{Dedekind's Theorem}\label{dedekind}

A theorem of Dedekind is frequently used to obtain information about 
the cycle structure of elements 
in a Galois group. 
The statement here is taken from van der Waerden \cite{vdw} but 
we have changed it to suit our context.

\begin{thm}\label{ded}
Let $f(x)$ be a monic polynomial of degree $n$ with coefficients in $\F_q[t]$.
Let ${\cal P}$ be a prime ideal in $\F_q[t]$, and let $k=\F_q[t]/{\cal P}$.
Suppose that 
\[
f(x) \equiv f_1(x) \cdots f_m(x) \ \textrm{ mod } {\cal P}
\]
where the $f_i(x)$ are distinct monic  irreducible polynomials in $k[x]$.
Let $d_i$ be the degree of $f_i(x)$.
Then the Galois group of $f(x)$ over $\F_q (t)$ contains a permutation that
is a product of disjoint cycles of lengths $d_1, \ldots , d_m$.
\end{thm}

Since every prime ideal ${\cal P}$ in $\F_q[t]$
is a principal ideal generated by an irreducible polynomial $p(t)\in \F_q[t]$,
the reduction $f(x)$ mod ${\cal P}$ can be obtained by replacing $t$ 
in each coefficient of $f(x)$ with a root $\beta \in \overline{\F_q}$ of $p(t)$.
We will be applying this theorem to $L(x)+tx$, with a judicious choice of $\beta$.
Therefore in our case  the reduction mod ${\cal P} $ is equal to  $L(x)+\beta x$.

\subsection{Ramification at $\infty$}

Let $f(x)\in \F_q[x]$ be separable and let $t$ be transcendental over $\F_q$.
The Galois group of $f(x)+t$ over $\F_q(t)$ is called the arithmetic monodromy group of $f$.
%This contains the geometric monodromy group as a normal subgroup, with cyclic quotient.
Because $f(x)$ is a polynomial, $f(x)+t$ is totally ramified at $\infty$.
 When the degree of $f(x)$ is relatively prime to $p$, the ramification at $\infty$ is tame
 and the inertia group of a place lying over $\infty$ is cyclic.
 We state this in the next lemma, which  follows from
 van der Waerden \cite{vdw2}, see also
the discussion in Section 2 of the paper by Cohen \cite{C}.

\begin{lemma}\label{n_cycle}
Let $f(x)\in \F_q[x]$ be a separable polynomial of degree $n$, where $\gcd (n,p)=1$.
Then the Galois group of $f(x)+t$ over $\F_q(t)$ contains an $n$-cycle.
\end{lemma}

We will apply this when $f(x)=L(x)/x$.

\subsection{Singer Cycles}

An element of $GL(n,q)$ of order $q^n-1$ is called a Singer cycle.
Subgroups of $GL(n,q)$ that contain Singer cycles
 are quite restricted in structure and we will use a theorem of Kantor 
 relating to Singer cycles that enables us to identify $\Gamma$ in certain circumstances when it contains a Singer cycle. We remark that it is easy to verify that $(L(x)+tx)/x$ is irreducible over
 $\mathbb{F}_q(t)$ and thus $q^n-1$ divides the order of the Galois group. This fact alone can help to pin down the structure of the Galois group.

We can understand the action of a Singer cycle on the $n$-dimensional space over $\mathbb{F}_q$ if we again invoke the basic theory
of finite fields. Let $F$ denote the finite field $\mathbb{F}_{q^n}$ of order $n$. $F$ is an $n$-dimensional vector space over $\mathbb{F}_q$,
which will serve as a model for our $n$-dimensional space. Let $z$ be an element of multiplicative order $q^n-1$ in $F$ and let
$C$ be the cyclic subgroup generated by $z$. We may call both $C$ and $z$ Singer cycles. Since the nonzero elements of $F$ are powers of $z$, $C$ acts transitively and regularly on the
nonzero elements of $F$. Consider the Frobenius automorphism $\tau$ of $F$ given 
by $\tau(\alpha)=\alpha^q$ for all $\alpha$ in $F$. It is clear
that $\tau$ normalizes $C$, and indeed, if we consider $C$ as a subgroup of $GL(n,q)$, the normalizer $N(C)$ of $C$ in $GL(n,q)$ is generated by
$C$ and $\tau$ and hence has order $n(q^n-1)$. 

\begin{lemma} \label{action_of_normalizer}
 Let $C$ be a Singer cycle in $GL(n,q)$ and let $N(C)$ be its normalizer. Suppose that some element $w$ of $N(C)$ fixes a nonzero vector
 in the underlying vector space over $\mathbb{F}_q$. Then $w$ has order dividing $n$.
 \end{lemma}
 
 \begin{proof}
 We have observed that $C$ acts transitively and regularly on the $q^n-1$ nonzero vectors. Thus, since $N(C)$ has order $n(q^n-1)$, the orbit-stabilizer theorem shows that the stabilizer of any nonzero vector has order $n$. This proves the lemma.
 \end{proof}

The following theorem of  Kantor \cite{K} classifies the subgroups of 
$GL(n,q)$ that contain a Singer cycle.

\begin{thm}\label{CK}
Let $G$ be a subgroup of $GL(n,q)$ that contains a Singer cycle.
Then $GL(n/d,q^d) \le G \le \Gamma L (n/d,q^d)$ for some divisor $d$ of $n$.
\end{thm}

We note that the case $d=1$ corresponds to $G=GL(n,q)$, and the case $d=n$
corresponds to $C\le G \le N(C)$ where $C$ is a Singer cycle.
We also note that if $n$ is a prime number, then these are the only possibilities.

 \subsection{Discriminant and Galois Theory}
 
 It will be useful for us to have a formula for the discriminant of a $q$-polynomial. The following must be well known but we provide a proof, as it is quite short.
 
 \begin{lemma} \label{discriminant_formula}
 Let $q$ be a power of the odd prime $p$ and let 
   $l(x)$ be a monic $q$-polynomial of $q$-degree $n\geq 1$ in $\mathbb{F}_q(t)[x]$. Suppose
 that the constant term $c$ of $G(x)=L(x)/x$ is nonzero. Then modulo squares in $\mathbb{F}_q(t)$, the discriminant of $G$ is
 \[
 (-1)^{(q^n-1)/2}c.
 \]
 
 \end{lemma}
 
 \begin{proof}
 We set $m=q^n-1$ and note that $m$ is even. Since $c$ is nonzero, it is easy to see that each root of $G$ has multiplicity one. Let
 the roots of $G$ be $\alpha_1$, \dots, $\alpha_m$. Then, if $D$ denotes the discriminant of $g$, it is well known that
 \[
 D=(-1)^{m(m-1)/2}\prod_{i=1}^m G'(\alpha_i).
 \]
 Now we have $xG(x)=L(x)$ and thus we may differentiate to obtain $G(x)+xG'(x)=L'(x)$. But $L'(x)=c$ since $L$ is a $q$-polynomial
 and thus we obtain
 \[
 \alpha_iG'(\alpha_i)=c
 \]
 for every root $\alpha_i$.
 
 Since $m$ is even, the product of the roots is $c$. We therefore obtain
 \[
 c\prod_{i=1}^m G'(\alpha_i)=c^m.
 \]
 When we recall that $m=q^n-1$, the rest of the proof is routine to check.
 
 \end{proof}
 
 Lemma \ref{discriminant_formula} shows that when $q$ is odd, the question of whether the Galois group of $q$-polynomial of $q$-degree $n$ is in the alternating group $A_{q^n-1}$ or not is determined by the coefficient of $x$. 
 We next show for the record that the situation is different when $q$ is even;
 the Galois group is almost always contained in the alternating group.
 
 \begin{thm} \label{almost_always_alternating}
 Let $q$ be a power of $2$ and let $F$ be a field of characteristic $2$ that contains $\F_q$. Let $L$ be a $q$-polynomial
 of $q$-degree $n$ in $F[x]$ with distinct roots. Then unless $q=n=2$, the Galois group of $L$ considered as a permutation
 group on the $q^n-1$ roots of $L(x)/x$ is in the alternating group $A_{q^n-1}$.
 \end{thm}
 
 \begin{proof}
 Let $\Gamma$ be the Galois group of $L$ over $F$. We know that $\Gamma$ may be considered to be a subgroup of $GL(n,q)$. Now apart
 from the case that $q=n=2$, $GL(n,q)$ has no normal subgroup of index 2. It follows that in the action of $GL(n,q)$ on $\F_q^n\setminus \{ 0 \}$,
 $GL(n,q)$ embeds into $A_{q^n-1}$. Thus, since $\Gamma$ is a subgroup of $GL(n,q)$, the same holds for $\Gamma$, except possibly
 when $q=n=2$.
 \end{proof}
 
 The case when $q=n=2$ is certainly exceptional. To illustrate this exceptionality, consider $L(x)=x^4+x$ in $\F_2[x]$.
 We have $L(x)=x(x+1)(x^2+x+1)$ in $\F_2[x]$, where $x^2+x+1$ is irreducible. Thus the Galois group has order 2 and is contained
 in $S_3$ but not $A_3$. 
 
 %The Berlekamp discriminant is used as a substitute for the ordinary discriminant in characteristic 2 but Lemma \ref{almost_always_alternating} shows that the Berlekamp discriminant is usually zero for a $q$-polynomial when $q$ is a power of 2.

%We will use the fact that for any polynomial of degree $m$ with
%coefficients in a field $K$ of characteristic different from 2, whose Galois group
%is naturally a subgroup of the symmetric group $S_m$, the Galois group is a 
%subgroup of the alternating group $A_m$ if and only if the discriminant is a square in $K$.

Finally for this section, we record
 the following Lemma  from elementary Galois theory which we will use.
It follows from Corollary 3.19 in \cite{Milne}.

\begin{lemma}\label{compositum}
Let $K$ and $L$ be finite-dimensional  extensions of a field $F$, that are both
contained in a field $M$. 
Suppose that the degrees $[K:F]$ and $[L:F]$ are both odd and at least one of $K$ and $L$ is a Galois extension of $F$.
Then $[KL:F]$ is odd, where $KL$ is  the compositum of $K$ and $L$ in $M$.
\end{lemma}

We will apply this lemma when $F$ is a finite field, and all finite extensions are
automatically Galois extensions.

\subsection{Theorem of G\"olo\u{g}lu-McGuire}

Our proof uses the following theorem from \cite{GMcG}.

\begin{thm}\label{gmg}
 Let $p$ be an odd prime and let $q=p^m$.
 Let $L(x)$ be a  $p$-linearized polynomial with coefficients in $\F_q$.
  Let $Im(L(x)/x)$ denote the image of $L(x)/x$, considered as a function
 $\F_q \longrightarrow \F_q$. Then
 \begin{equation*}\label{ourcond}
 Im(L(x)/x) \subseteq  \{ \beta^2 \ : \ \beta \in \F_q^* \} \cup \{0\}
 \end{equation*} 
 if and only if  $L(x) = a x^{p^d}$ for some $a\in \F_q^*$ which is a square, and some $0\le d < m$.
\end{thm}

\section{Proof of Main Theorem}

In this section we present the proof of Theorem \ref{main}.
We assume that $q$ is odd and that $n$ is an odd prime, although some parts of 
the argument hold in more generality.

 \setcounter{thm}{1}

\begin{thm}
Let $q$ be odd, and let $n$ be an odd prime. 
Let $L(x)\in \F_q[x]$ be a monic $q$-polynomial of $q$-degree $n$.
Then the Galois group of $L(x)+tx \in \F_q(t)[x]$ is isomorphic to $GL(n,q)$, unless $L(x)=x^{q^n}$,
in which case the Galois group is isomorphic to $\Gamma L (1,q^n)$.
\end{thm}

\begin{proof}
Let $\Gamma$ be the Galois group of $L(x)+tx$ over $\F_q(t)$.
As we stated in the introduction, $\Gamma$ is also the Galois group of $(L(x)+tx)/x$.
By Lemma \ref{n_cycle},
$\Gamma$ contains a $(q^n-1)$-cycle, i.e., a Singer cycle.
Denote the subgroup generated by the Singer cycle as $C$. 
Let $N(C)$ be the normalizer of $C$ in $GL(n,q)$.
By Theorem \ref{CK} and the comments after, 
 either  $\Gamma=GL(n,q)$, or $C \le \Gamma \le N(C)$.

If $L(x)=x^{q^n}$ then  it is easy to check that the Galois group of
$(L(x)+tx)/x=x^{q^n-1}+t$ is $\Gamma L (1,q^n)$,
 a semidirect product of  a cyclic  group of order $q^n-1$
with a cyclic group of order $n$, so $\Gamma = N(C)$.
Therefore, showing that $C \le \Gamma \le N(C)$ implies $L(x)=x^{q^n}$
will complete the proof of Theorem \ref{main}.
Let us assume from now on that $C \le \Gamma \le N(C)$. 
We will show that this implies $L(x)=x^{q^n}$.

Without loss of generality we may assume that the coefficient of $x$ in $L(x)$ is 0.

Choose an integer $k>n$ and consider the field $\F_{q^k}$.
For $\alpha \in \F_{q^k}$ we specialize $t$ to be $-L(\alpha)/\alpha$,
as described in Section \ref{dedekind}.
Call the resulting polynomial $L_\alpha(x)$, i.e., define
\[
L_\alpha(x):=L(x)-\frac{L(\alpha)}{\alpha} x.
\]
Then $L_\alpha(x)$ is a monic $q$-polynomial of $q$-degree $n$ in $\F_{q^k}[x]$.
Note that $L_\alpha(\alpha)=0$.
The roots of $L_\alpha(x)$ are distinct if $L(\alpha)\not=0$.
Since $k>n$ there do exist $\alpha \in \F_{q^k}$ with $L(\alpha)\not=0$,
and we choose any such $\alpha$.

Suppose that the monic irreducible factors of $L_\alpha(x)/x$ over
$\F_{q^k}[x]$ have degrees $d_1, \ldots , d_m$.
Since $L_\alpha(x)/x$ has at least one root in $\F_{q^k}$, namely $\alpha$,
at least one of the $d_i$ is equal to 1.
By  Theorem \ref{ded}, $\Gamma$ has an element $\sigma$ whose action on the roots
of $(L(x)+tx)/x$ consists of disjoint cycles of lengths $d_1, \ldots , d_m$.
Since one of the $d_i$ is 1, $\sigma$ fixes a point.
By Lemma  \ref{action_of_normalizer}, $\sigma$ has order dividing $n$
because we are assuming  that $\Gamma \le N(C)$. 
Since $n$ is odd,  the order of $\sigma$ is odd.
Because the order of $\sigma$ is the lowest common multiple of $d_1, \ldots , d_m$,
this implies that each $d_i$ is odd. 
By Lemma \ref{compositum}, we conclude that the Galois
group of $L_\alpha(x)/x$ over $\F_{q^k}$ has odd order,
since the splitting field is a compositum of  extensions of $\F_{q^k}$
of degrees $d_1, \ldots , d_m$.
This Galois group is cyclic, being the Galois group of a finite extension 
of a finite field.

We have now shown that the Galois
group of $L_\alpha(x)/x$ over $\F_{q^k}$ is cyclic of odd order. 
Because an odd length cycle is an even permutation,
 the Galois group of $L_\alpha(x)/x$ over $\F_{q^k}$
consists entirely of even permutations, in the permutation action on the roots.
By elementary Galois theory, 
the discriminant of $L_\alpha(x)/x$ is a square in $\F_{q^k}$.

Next we calculate the discriminant.
The constant term in $L_\alpha(x)/x$ is $L(\alpha)/\alpha$.
By Lemma \ref{discriminant_formula} the discriminant of 
$L_\alpha(x)/x$ is equal to $L(\alpha)/\alpha$ if $q\equiv 1\pmod{4}$,
and is equal to $-L(\alpha)/\alpha$ if $q\equiv 3\pmod{4}$, modulo squares.

First suppose  $q\equiv 1\pmod{4}$.
If $\alpha \in \F_{q^k}$ then either $L(\alpha)/\alpha =0$
or $L(\alpha)/\alpha$ is a square in $\F_{q^k}$.
We apply  Theorem \ref{gmg}
 to $L(x)/x$ as a function on $\F_{q^k}$,
when $L(x)$ is considered as a $p$-polynomial (any $q$-polynomial is a $p$-polynomial).
We conclude that $L(x)=ax^{p^d}$ as a function on $\F_{q^k}$,
for some $a\in \F_q^*$ which is a square, and some $0\le d < mk$, where $q=p^m$.
This implies that, as elements of the polynomial ring $\F_{q^k}[x]$, 
the polynomial $L(x)-ax^{p^d}$ is divisible by $x^{q^k}-x$.
However we chose $k>n$, so $q^k$ is greater than the degree of the polynomial $L(x)-ax^{p^d}$.
This implies that $L(x)-ax^{p^d}$ is the zero polynomial, and so
$L(x)=ax^{p^d}$ as polynomials. Since we assumed $L(x)$ is monic of $q$-degree $n$,
we must have $L(x)=x^{q^n}$.

If $q\equiv 3\pmod{4}$ and $\alpha \in \F_{q^k}$ then either $L(\alpha)/\alpha =0$
or $-L(\alpha)/\alpha$ is a square in $\F_{q^k}$.
By Theorem \ref{gmg}
applied to $-L(x)/x$ as a function on $\F_{q^k}$, we may conclude that $L(x)=ax^{p^d}$ 
as a function on $\F_{q^k}$.
The argument is completed in the same way as when $q\equiv 1\pmod{4}$.

This completes the proof of Theorem \ref{main}.
\end{proof}

Theorem \ref{main} determines the arithmetic monodromy group.
We remark that the geometric monodromy group is also $GL(n,q)$, 
except for the case $L(x)=x^{q^n}$ when the geometric monodromy group is
cyclic of order $q^n-1$.

\section{Characteristic 2}\label{char2}

We have no reason to believe that Theorem \ref{main} does not hold when $q$ is a power of 2, but one of our main tools
used in the odd characteristic proof, Theorem \ref{gmg}, does not apply. We briefly show here how our earlier methods can be applied
in characteristic 2, yielding a strong partial result.

 \setcounter{thm}{10}

\begin{thm}\label{main2}
Let $q$ be a power of $2$, and let $n$ be an odd prime. 
Let $L(x)\in \F_q[x]$ be a monic $q$-polynomial of $q$-degree $n$ with
\[
L=a_1x^q+a_2x^{q^2}+\cdots+a_{n-1} x^{q^{n-1}}+x^{q^n}.
\]
Suppose that
\[
a_1+\cdots +a_{n-1}+1\neq 0.
\]
Then the Galois group of $L(x)+tx \in \F_q(t)[x]$ is isomorphic to $GL(n,q)$, unless $L(x)=x^{q^n}$,
in which case the Galois group is isomorphic to $\Gamma L (1,q^n)$.
\end{thm}

\begin{proof}
%We assume that $L(x)$ is not $x^{q^n}$.
Let $\alpha$ be any nonzero element of $\F_q$. Then we find that
\[
L(\alpha)=\alpha(a_1+\cdots +a_{n-1}+1)\neq 0.
\]
We set
\[
L_\alpha(x):=L(x)-\frac{L(\alpha)}{\alpha} x
\]
and note as before that $\alpha$ is a root of $L_\alpha$.

Following the proof of Theorem \ref{main}, 
 it suffices to assume that the Galois group of $L(x)+tx$ is contained
in the normalizer of a Singer cycle, and prove that $L(x)=x^{q^n}$. 
The roots of $L_\alpha(x)$ in its splitting field are all distinct,
and as shown in the proof of Theorem \ref{main}, it follows that
all monic irreducible factors of $L_\alpha(x)$ in $\F_q [x]$ have degree dividing $n$.

It is well known that the polynomial $x^{q^n}-x$ is the product of all monic irreducible polynomials of degree dividing $n$ over $\F_q$.
Since $L_\alpha(x)$ has the same degree as $x^{q^n}-x$, we must have
$L_\alpha(x)=x^{q^n}-x$.
But this forces $L(x)=x^{q^n}$.
\end{proof}

The proof above works for all $q$, not just powers of 2. It provides a strong partial result in all 
characteristics. Other specializations provide even stronger partial results.

{\bf Acknowledgement}

The authors thank Marco Timpanella for helpful conversations.

\end{document}